\documentclass[10pt]{amsart}
\usepackage{hyperref}
\usepackage{amsmath,amssymb,amsfonts,latexsym,amsthm}

\usepackage{mathrsfs}

\numberwithin{equation}{section}

\usepackage{graphicx}

\usepackage{ifthen} 

\provideboolean{shownotes} 
\setboolean{shownotes}{true} 
\newcommand{\margnote}[1]{
\ifthenelse{\boolean{shownotes}}%
{\marginpar{\raggedright\tiny\texttt{#1}}}%
{}%
}

\newcommand{\hole}[1]{
\ifthenelse{\boolean{shownotes}}%
{\begin{center} \fbox{ \rule {.25cm}{0cm}
\rule[-.1cm]{0cm}{.4cm} \parbox{.85\textwidth}{\begin{center}
\texttt{#1}\end{center}} \rule {.25cm}{0cm}}\end{center}}
{}
}

\theoremstyle{plain}

\newtheorem{lemma}{Lemma}[section]
\newtheorem{theorem}[lemma]{Theorem}

\theoremstyle{definition}

\newtheorem{remark}[lemma]{Remark}
\newtheorem{definition}[lemma]{Definition}

\theoremstyle{remark}

\newcommand{\R}{\mathbb{R}}

\newcommand{\cD}{{\mathcal{D}}}

\newcommand{\cE}{{\mathcal{E}}}

\newcommand{\cM}{{\mathcal{M}}}

\begin{document}

\title[On the role of diffusion in the tumor growth paradox]{On the role of cancer cells' diffusion in the tumor growth paradox}

\author[I. Padilla]{Isai Padilla}
 
\address{{\rm (I. Padilla)} Posgrado en Ciencias Matem\'{a}ticas\\Universidad Nacional Aut\'{o}noma 
de M\'{e}xico\\Circuito Exterior s/n, Ciudad de M\'{e}xico C.P. 04510 (Mexico)}

\email{ipbepsilon@gmail.com}

\author[R.G. Plaza]{Ram\'on G. Plaza$^*$}
\thanks{$^*$Author to whom correspondence should be addressed}
\address{{\rm (R. G. Plaza)} Instituto de 
Investigaciones en Matem\'aticas Aplicadas y en Sistemas\\Universidad Nacional Aut\'onoma de 
M\'exico\\Circuito Escolar s/n, Ciudad de M\'{e}xico C.P. 04510 (Mexico)}

\email{plaza@mym.iimas.unam.mx}

\begin{abstract}
In this contribution, the non-local, integro-partial differential system of equations proposed by Hillen \textit{et al.} \cite{HEH13} to account for the \emph{the tumor growth paradox} (or the observation that some incomplete cancer treatments may enhance tumor growth) is reviewed. It is shown that when cancer cells' diffusion effects and Neumann boundary conditions are taken into consideration, the same paradoxical tumor growth emerges. 
\end{abstract}



\maketitle

\setcounter{tocdepth}{1}



\section{Introduction}
\label{secintro}

Cancer stem cells (CSCs) have been identified in many types of cancer such as leukaemia \cite{Cobetal00,LSVetal94}, and carcinomas of breast \cite{AlHetal03}, colon \cite{OPGD07}, brain \cite{SCTetal03} and pancreas \cite{Lietal07}, among others. The clinical observation that some incomplete cancer treatments may enhance tumor growth is known in the literature as \emph{the tumor growth paradox}. It has been suggested (see, e.g., Enderling \textit{et al.} \cite{EACBHH}) that the presence of CSCs may explain the tumor growth paradox thanks to their pluripotency and their resistance to treatment. Therefore, some CSCs-based mathematical models have been proposed in recent years to account for the tumor growth paradox. In this paper, we review the non-local, integro-partial differential system of equations proposed by Hillen, Enderling and Hahnfeldt \cite{HEH13}, a model which considers an heterogeneous population of CSCs and standard, non-stem, cancer cells (CCs). The model system is endowed with initial conditions for both cell populations and boundary conditions of Dirichlet or Neumann type. The authors in \cite{HEH13} show that, when diffusion effects are neglected, the resulting purely dynamical system underlies a sort of tumor growth paradoxical behaviour as a result of the immune response from cytotoxic treatments. Their analysis is based on a direct application of geometrical singular perturbation theory for ordinary differential equations (ODEs) \cite{Fen71, Fen79, Hek10,Jon95}.  In particular, they find that, among the equilibrium states of the reduced ODE system, the only global attractor is the pure CSC state, so that after a sufficiently long time the tumor will consist of CSCs only. In later contribution, Maddalena \cite{Madd14} showed that when diffusion is present and under Neumann boundary conditions, stationary solutions may change. Notably, Maddalena proved that the equilibrium states of the associated ODE are the only stationary solutions to the PDE system under Neumann conditions, but they change their stability due to diffusion effects. If diffusion is present and Dirichlet boundary conditions are considered, Delgado \textit{et al.} \cite{DeDS18} recently proved that there exist non-trivial stationary states that include coexistence states of CSCs and CCs with positive values for both populations. These works \cite{Madd14,DeDS18}, however, do not discuss the emergence of paradoxical growth behavior.

The motivation of the present review is to analyze the possible effects of cancer cells' diffusion on the tumor growth paradox. If diffusion effects are taken into consideration, does a paradoxical tumor growth emerge? We discuss the possible effects of diffusion by applying the invariant foliation theory of Bates, Lu and Zeng \cite{BaLZ98, BaLZ00}, which can be seen as the geometric singular perturbation counterpart theory for PDEs with a fast/slow structure (with some limitations). We consider the simplest case of Neumann boundary conditions: since the equilibria for the ODE coincide with stationary solutions to the PDE,  the ``slow manifold" on which the long time dynamics takes place is essentially the same as the one described by Hillen \textit{et al.} \cite{HEH13} for the ODE system. We show that the very same conditions for the emergence of the tumor growth paradox can be retrieved for the semiflow of the diffusive system under Neumann conditions. The description of the paradoxical tumor growth under the flow of the full PDE system compensates, we hope, for the lack of novelty of the results.

\section{The tumor growth paradox}
\label{secmodel}

\subsection{The non-local model of Hillen, Enderling and Hahnfeldt}

Hillen \textit{et al.} \cite{HEH13} proposed the following coupled integro-differential, no-local system of partial differential equations (PDEs) to describe the CSC and CC dynamics:
\begin{equation}
\label{HEHmodel}
\begin{aligned}
u_t &= D_u \Delta u + \delta \gamma \int_{\Omega} k(x,y,p(x,t)) u(y,t) \, dy, \\
v_t &= D_v \Delta v + (1-\delta) \gamma \int_{\Omega} k(x,y,p(x,t)) u(y,t) \, dy - \bar{\alpha} v + \\ &\;\;\;+ \rho \int_\Omega k(x,y,p(x,t)) v(y,t) \, dy,
\end{aligned}
\end{equation}
for $x \in \Omega \subseteq \R^n$ and $t \geq 0$. $\Omega$ is a open, bounded set with smooth boundary $\partial \Omega$. Here $u = u(x,t)$ denotes the density of CSCs and $v = v(x,t)$ denotes the density of non-stem cancer cells (CCs), at each point $x \in \Omega$ and at time $t > 0$. $p (x,t) = u(x,t) + v(x,t)$ is the total tumor cell density and $\delta \in [0,1]$, $\gamma > 0$ and $\rho > 0$ are constant parameters. The coefficient $\delta$ is the fraction of symmetric divisions of a CSC: if $\delta = 0$ it divides into one CSC and one normal CC, whereas if $\delta = 1$ the CSC divides into two CSCs; both divisions take place at mitosis rate $\gamma$ (number of cell cycles per unit time). $\rho > 0$ is the number of cycles per unit time for the normal CCs. The function $k = k(x,y,p)$ is an integral kernel that describes the rate of cell divisions contributing to the point in space $x$ from a cell at location $y$ per cell cycle time. $D_v > 0$ and $D_u > 0$ are the diffusion coefficients of CCs and CSCs, respectively. The parameter $\bar{\alpha} > 0$ denotes the death rate of normal CCs. Notice that in model \eqref{HEHmodel} the CSCs are considered to be inmortal. (See \cite{HEH13} for further information about this hypothesis.)

Model equations \eqref{HEHmodel} can be derived from a stochastic process known as a \textit{birth-jump process} as introduced in \cite{HGMdV15} in the context of forest fire spotting. Although the authors in \cite{HEH13} do not discuss this derivation, the reader is referred to \cite{PadTh19} (in the CSC and CC dynamics setting) for a complete derivation departing from a single-cell agent-based model in the spirit of Enderling \textit{et al.} \cite{EACBHH}.

System \eqref{HEHmodel} is also endowed with boundary conditions of Neumann or Dirichlet type. Homogeneous Neumann boundary conditions read
\begin{equation}
\label{Neumannbc}
\frac{\partial u}{\partial \hat{\nu}} = 0, \quad \frac{\partial v}{\partial \hat{\nu}} = 0, \qquad \text{on } \; \partial \Omega,
\end{equation}
where $\hat{\nu}$ is the normal exterior unit normal at $\partial \Omega$. They represent no biological flux of cancer cells due to impenetrable physical constraints such as tissue surrounded by bone or membranes. In this case the kernel $k$ must satisfy the condition
\[
k(x,y,p) = 0, \quad \text{for all } \; \; x \notin \Omega,
\]
inasmuch as there cannot be progeny contribution from cells located outside of the domain. On the other hand, homogeneous Dirichlet boundary conditions have the form
\begin{equation}
\label{Dirichletbc}
u = 0, \quad v = 0, \qquad \text{on } \; \partial \Omega,
\end{equation}
and they model tissues where cells can leave but not re-enter again, for example, in the case of vascularized carcinomas; thus, the transport of cells out of the domain but without allowing re-entering require the kernel $k$ to satisfy
\[
k(x,y,p) = 0, \quad \text{for all } \; \; y \notin \Omega.
\]

Finally, one should impose initial conditions of the form
\begin{equation}
\label{initialc}
u(x,0) = u_0(x), \quad v(x,0) = v_0(x), \qquad x \in \Omega.
\end{equation}

\subsubsection*{First reduction} Hillen \textit{et al.} \cite{HEH13} consider a first simplification of model \eqref{HEHmodel}, which consists of assuming that the progeny placement depends only on the total density at the destination, namely, that
\[
k = k(p(x,t)).
\]
The result is the following non-local integro-differential system,
\begin{equation}
\label{HEHm}
\begin{aligned}
u_t &= D_u \Delta u + \delta \gamma k(u+v) \int_{\Omega} u(y,t) \, dy, \\
v_t &= D_v \Delta v + (1-\delta) \gamma k(u+v)  \int_{\Omega} u(y,t) \, dy - \bar{\alpha} v + \rho k(u+v) \int_\Omega v(y,t) \, dy,
\end{aligned}
\end{equation}
subject to boundary conditions of Neumann \eqref{Neumannbc} or Dirichlet type \eqref{Dirichletbc} and to appropriate intitial conditions \eqref{initialc}. This is the system of equations we are concerned with.

For both analytical and numerical studies one must write the system in non-dimensional form. In this fashion, one is able to isolate the relevant parameters. Since the parameter $\bar{\alpha} > 0$ will be a key to differentiate CCs from CSCs, we shall scale time with the mitosis rate. For simplicity and following \cite{HEH13} we shall assume that both CCs and CSCs have the same proliferation rate (there is evidence that this is not the case \cite{Toletal14}, though), that is,
\[
\gamma = \rho.
\]
Notice that $\gamma$ has physical units of frequency. Thus, making the substitutions ${\alpha} = \bar{\alpha} / \gamma$, $d = D_u / D_v$ and
\[
x \to \Big(\frac{\gamma}{D_v}\Big)^{1/2} x, \quad t \to \gamma t, \quad u \to \frac{u}{U}, \quad v \to \frac{v}{U}, \quad k \to \Big( \frac{D_v}{\gamma} \Big)^{n/2} k,
\]
where  $U$ is a characteristic tumor cell density, we arrive at the non-dimensional system
\begin{equation}
\label{HEHnd}
\begin{aligned}
u_t &= d \Delta u + \delta  k(u+v) \int_{\Omega} u(y,t) \, dy, \\
v_t &=  \Delta v + (1-\delta)  k(u+v)  \int_{\Omega} u(y,t) \, dy - {\alpha} v +  k(u+v) \int_\Omega v(y,t) \, dy.
\end{aligned}
\end{equation}

Notice that $\alpha = \bar{\alpha}/\gamma$ is now the ratio between the CCs death rate and the tumor cell proliferation rate. If $0 < \alpha \leq 1$ tumor cells are born at a higher rate than the rate in which the cytotoxic therapy kills them. If $\alpha > 1$ the treatment kills more cells than those that are born per cycle time. Thus, $\alpha$ is a measure of the effectiveness of the treatment.

Finally, it is assumed that the function $k = k(p)$ satisfies
\begin{equation}
\label{hypk}
\left\{
\begin{aligned}
& k(p) \, \text{ is piecewise differentiable, } \\
& k(p) > 0,  \; \text{ for } \; p \in [0,1),\\
& k(p) = 0, \;\; \text{for } \; p \geq 1,\\
& k(p) \, \text{ is decreasing for } \; p \in [0,1).
\end{aligned}
\right.
\end{equation}
The typical form of $k$ considered in \cite{HEH13} reads,
\begin{equation}
\label{ksig}
k(p) = \max \{ 1 - p^\sigma, 0\}, \quad \sigma \geq 1.
\end{equation}

\subsection{Tumor growth paradox}

Hillen \textit{et al.} \cite{HEH13} define paradoxical tumor growth as follows.
\begin{definition} 
\label{defparadox}
Let $p_\alpha(t)$ be the tumor population with spontaneous death rate $\alpha$ for CCs at time $t \geq 0$. The population exhibits a \textit{tumor growth paradox} if there exist death rates $\alpha_1 < \alpha_2$ and positive times $t_1, t_2$ and $T_0$ such that
\[
p_{\alpha_1}(t_1) = p_{\alpha_2}(t_2) \quad \text{and } \quad p_{\alpha_1}(t_1 + T) < p_{\alpha_2}(t_2 + T), \quad \text{for each } \, T \in (0,T_0).
\]
\end{definition}

In other words, there is paradoxical tumor growth whenever, for tumors initially of the same size at some point of their evolution, the overall tumor size increases despite a higher CCs death rate.

\subsubsection*{Second reduction} The authors in \cite{HEH13} perform a further simplification by assuming that tumor growth is uniform across the domain, that is, that $u$ and $v$ do not depend on the spatial variable $x \in \Omega$. Thus, diffusion effects are neglected. By defining mean densities as $\bar{u}(t) = |\Omega| u(t)$, $\bar{v}(t) = |\Omega| v(t)$ and $\overline{p}(t) = |\Omega| p(t)$, model \eqref{HEHnd} reduces to the ODE system
\begin{equation}
\label{ODEs}
\begin{aligned}
\frac{d \bar{u}}{dt} &= \delta k(\overline{p}) \bar{u}, \\
\frac{d \bar{v}}{dt} &= (1-\delta) k(\overline{p}) \bar{u} - \alpha \bar{v} + k(\overline{p}) \bar{v}.
\end{aligned}
\end{equation}
Initial conditions \eqref{initialc} are substituted by constant initial conditions for the ODE of the form $(\bar{u}, \bar{v})(0) = (u_0,v_0)$.

Upon inspection of the associated ODE system \eqref{ODEs} one finds that its equilibrium points are
\begin{equation}
\label{states}
\begin{aligned}
P_0 = (0,0), & & \text{(cancer dissapears),}\\
P_1= (0, k^{-1}(\alpha)) = (0, v_*(\alpha)), & & \text{(pure CCs state),} \\
P_2 = (k^{-1}(0), 0) = (1,0), & & \text{(pure CSCs state),}\\
\end{aligned}
\end{equation}
where $v_*$ is the solution to $\alpha = k(v_*)$. Notice, however, that this may not be well defined in the case where $\alpha > 1$ because, for example, under hypotheses \eqref{hypk} the range of $k = k(p)$ is contained in $[0,1]$. Via linearization of the system of equations around equilibrium points, the authors find that
\begin{itemize}
\item $P_0$ is a saddle if $\alpha > k(0) = 1$ and an unstable node if $0 < \alpha < 1$.
\item $P_1$ is a saddle whenever $v_*(\alpha)$ exists, and,
\item $P_2$, the pure stem cell state, is a stable node or stable spiral.
\end{itemize}
Whence, the authors focus on the role of $P_2$ as the only globally asymptotic equilibrium point in the positively invariant triangular region
\[
\tilde{R} = \{ (\bar{u},\bar{v}) \, : \, \bar{u} \in [0,1], \, \bar{v} \geq 0, \, \bar{u} + \bar{v} \leq 1\}.
\]
That is, $P_2 = (1,0)$ is the only global attractor and after a sufficiently long time the tumor will consist of CSCs only. 

By considering a small parameter regime, namely, for $0 < \delta \ll 1$, sufficiently small, which is tantamount to consider very few symmetric divisions and mostly asymmetric ones (CSCs splitting into one CSC and one CC), the authors in \cite{HEH13} analyze the fast/slow structure associated to system \eqref{ODEs} by means of a geometric singular perturbation analysis \cite{Hek10,Jon95}. The result is the identification of a slow manifold, given by
\[
\widetilde{\cM}_0 = \{ (\bar{u},\bar{v}) \in \tilde{R} \, : \, k(\bar{p}) \bar{p} = \alpha \bar{v}, \,\; \bar{p} = \bar{u} + \bar{v} \},
\]
which is normally hyperbolic under the flow of the fast system (which is retrieved from setting $\delta = 0$ in \eqref{ODEs}). By Fenichel's theorems \cite{Hek10,Jon95}, there exists an invariant manifold $\widetilde{\cM}_\delta$ for the full system, close to $\widetilde{\cM}_0$, that can be written as a graph on $\widetilde{\cM}_0$. In this fashion, the long time dynamics is determined by the solution to the ``outer" system on the slow manifold. Once the solution has settled onto the slow manifold, the paradoxical tumor growth emerges. More precisely, the authors prove the following
\begin{theorem}[tumor growth paradox; Hillen \textit{et al.} \cite{HEH13}]
Let $(\bar{u}_1, \bar{v}_1)$ and $(\bar{u}_2, \bar{v}_2)$ be the corresponding solutions to system \eqref{ODEs} for values of $\alpha_1$ and $\alpha_2$, respectively, such that $\alpha_1 > \alpha_2 > 0$. Assume that the tumor dynamics has settled into the slow manifold $\widetilde{\cM}_0$ and that for a certain time $t_0 > 0$ the tumors have the same size, $\bar{p}_1(t_0) = \bar{p}_2(t_0) \in (0,1)$. Then $(d/dt) \bar{p}_1(t_0) > (d/dt) \bar{p}_2(t_0)$ and $\bar{p}_1(t) > \bar{p}_2(t)$ for all $t > t_0$. Moreover, if the tumors have the same initial conditions, then there exist times $t_a, t_b$ such that $\bar{p}_1(t_a)=\bar{p}_2(t_b)$ and $\bar{p}_1(t_a + \theta) > \bar{p}_2(t_b + \theta)$ for each $\theta > 0$.
\end{theorem}

For details, see Theorems 3.3 and Corollary 3.4 in \cite{HEH13}. In this fashion, paradoxical tumor growth establishes itself as a robust and typical property of the ODE model \eqref{ODEs}. 

Other authors have studied model \eqref{HEHnd}. For instance, Borsi \textit{et al.} \cite{BFPH17} established the existence of spatially dependent solutions, $(u,v)(x,t)$ to the non-local model without diffusion ($D_u = D_v = 0$), and Fasano \textit{et al.} \cite{FMP16} proved the well-posedness of system \eqref{HEHm} under Dirichlet boundary conditions. However, we review (and focus on) the results by Maddalena \cite{Madd14} and Delgado \textit{et al.} \cite{DeDS18} pertaining to the existence of stationary states under the influence of diffusion.

\subsection{Stationary states and Neumann boundary conditions}

In a later contribution, Maddalena \cite{Madd14} considered the complete model with diffusion \eqref{HEHnd} under Neumann boundary conditions \eqref{Neumannbc} and initial conditions \eqref{initialc}. First, the author establishes the existence of solutions applying classical arguments (cf. \cite{Amann76}). 
\begin{theorem}[Maddalena \cite{Madd14}]
For any $(u_0, v_0) \in H^2(\Omega) \times H^2(\Omega)$ there exists $T > 0$ and a solution $(u,v) \in C([0,T);H^2(\Omega) \times H^2(\Omega)) \cap C^1((0,T); H^2(\Omega) \times H^2(\Omega))$ to the Cauchy prpblem for \eqref{fast} under Neumann boundary conditions \eqref{Neumannbc}.
\end{theorem}
Actually, by standard semigroup theory \cite{EN06,Kat75,Pazy83}, it is possible to show that the solution operator to \eqref{fast} with Neumann conditions constitute a semiflow on $L^2(\Omega) \times L^2(\Omega)$, that is, a $C_0$-semigroup $S_\delta(t)$ (indexed by $\delta \in [0,1]$). For the proof, one applies a general result by Kato \cite{Kat75} (see also Section \ref{secNeumi} below). Details are omitted.
\begin{theorem}
For each initial condition $(u_0, v_0) \in L^2(\Omega) \times L^2(\Omega)$, there exists $T > 0$ such that the Cauchy problem for system \eqref{fast}, subject to Neumann boundary conditions \eqref{Neumannbc} has a unique solution $(u,v) \in C((0,T); \cD \times \cD) \cap C^1([0,T); L^2(\Omega) \times L^2(\Omega))$ (with $\cD = \{ u \in C^\infty(\Omega) \, : \, \mathrm{tr} \, \partial_\nu u = 0 \, \text{on } \, \partial \Omega\}$ dense in $L^2(\Omega)$) which we denote as
\[
(u, v) = S_\delta(t) (u_0, v_0).
\]
Moreover, the family of operators $S_\delta(t) : L^2(\Omega) \times L^2(\Omega) \to L^2(\Omega) \times L^2(\Omega)$ constitute a $C_0$-semigroup for each $\delta \in [0,1]$.
\end{theorem}

It is also shown in \cite{Madd14} that the region $R = [0, 1] \times [0, k^{-1}(\alpha)]$ is invariant under the flow $S_\delta(t)$ (provided that $k^{-1}(\alpha)$ exists: that is, for $\alpha < 1$ under hypotheses \eqref{hypk}). One of the main observations in \cite{Madd14} is the non-existence of spatially inhomogeneous steady states.
\begin{lemma}[Maddalena \cite{Madd14}]
\label{lemMadda}
All stationary solutions $(u,v)(x)$ to system \eqref{HEHnd} with Neumann boundary conditions \eqref{Neumannbc} are constant states.
\end{lemma}
\begin{remark}
\label{remXXX}
The strategy to prove this result relies on functional inequalities or so called ``entropy methods". Indeed, by considering an energy functional of the form
\[
\cE(t) = \int_\Omega (\nabla u_t \cdot \nabla u + \nabla v_t \cdot \nabla v) \, dx
\]
for solutions $(u,v)$ to \eqref{HEHnd} under Neumann conditions, Maddalena \cite{Madd14} shows that $\lim_{t \to \infty} \cE(t) = 0$, precluding the existence of spatially inhomogeneous attractors for the Neumann semiflow associated to \eqref{HEHnd}. This implies, in turn, that the only equilibrium solutions to the PDE system are constant steady states of the associated ODE system. 
\end{remark}

Consequently, upon inspection of \eqref{ODEs}, one finds that the only stationary solutions to \eqref{HEHnd} under Neumann conditions are the states \eqref{states}. By performing a linearized stability analysis of the steady states \eqref{states} as solutions to system \eqref{HEHnd} under Neumann conditions, Maddalena establishes that: $P_1 = (1,0)$ and $P_2 = (0, v_*(\alpha))$ are asymptotically stable (stable nodes), whereas $P_0 = (0,0)$ is a stable node provided that $- \mu_1 d + k(0) < 0$ and $- \mu_1 + k(0) - \alpha < 0$, where $\mu_1 = \mu_1(\Omega) > 0$ is the first non-zero eigenvalue of the Neumann Laplacian, $- \Delta_N$ in $\Omega$. Otherwise $P_0$ is a saddle. This is the main result of the analysis in \cite{Madd14}: the action of diffusion changes the stability properties of the very same equilibrium states as in the ODE reduction. The author, however, does not examine the rise of paradoxical tumor growth near $P_1$, for instance.

\subsection{Coexistence states under Dirichlet boundary conditions}

Delgado \textit{et al.} \cite{DeDS18} recently studied the stationary version of model \eqref{HEHm}, which is the non-local elliptic PDE system
\begin{equation}
\label{HEHstat}
\begin{aligned}
0 &= D_u \Delta u + \delta \gamma k(u+v) \int_{\Omega} u(y,t) \, dy, \\
0 &= D_v \Delta v + (1-\delta) \gamma k(u+v)  \int_{\Omega} u(y,t) \, dy - \bar{\alpha} v + \rho k(u+v) \int_\Omega v(y,t) \, dy,
\end{aligned}
\end{equation}
for $x \in \Omega$, under Dirichlet boundary conditions \eqref{Dirichletbc}. The authors in \cite{DeDS18} show that there are non-trivial stationary solutions with positive components $u$ and $v$ for the non-local elliptic system \eqref{HEHstat}. These solutions correspond to \textit{coexistence} states for the evolutionary system of equations \eqref{HEHm} (and of the normalized system \eqref{HEHnd}, of course). Unlike systems \eqref{HEHnd} with Neumann boundary conditions and the reduced ODE system \eqref{ODEs} of Hillen and co-authors, under Dirichlet boundary conditions there exist steady states for which both populations of CCs and CSCs may coexist. For instance, if $\delta < 1$ is fixed and $\alpha$ is increased, these non-trivial equilibria arise.  Their analysis is based on bifurcation theory in the spirit of Crandall and Rabinowitz \cite{CrR71} and fixed point index theory in conic domains. It is to be observed, however, that the authors do not analyze the stability of such states under the flow of the parabolic system of equations of evolution \eqref{HEHm}.

Under the light of these results, we turn our attention to the question whether there is paradoxical tumor growth when diffusion is switched on. The answer must take into account the choice of boundary conditions and must certainly relate to the extension of geometrical singular perturbation theory to semiflows (i.e. to PDEs), known as invariant foliation theory. In the sequel we analyze the easiest case where, under Neumann conditions, the only stationary solutions are constant equilibria of the associated ODE.

\section{Application of invariant foliation theory}
\label{secfoliation}

\subsection{The invariant manifold theorems}

The extension of geometrical singular perturbation theory to infinite dimensional spaces (e.g., to PDEs) is quite difficult. There are, however, some extensions of Fenichel theorems to semiflows which can be interpreted as the first step for a geometrical singular perturbation theory for infinite-dimensional systems with a fast/slow structure.

Consider a Banach space $X$ and a (fast) semiflow $S_0(t)$ on $X$. It is assumed the existence of a $C^1$ compact connected manifold $\cM_0 \subset X$, invariant under the fast semiflow $S_0(t)$, which is normally hyperbolic with respect to the fast semiflow (see Remark \ref{RemiMafia}). For any mapping $F$ on a bounded subset $B \subset X$ we define $\| F \|_0 := \sup \{ \| F(x)\|_X \, : \, x \in B\}$ and $\| F \|_1 := \| F \|_0 + \| DF \|_0$. The following theorem, due to Bates, Lu and Zeng \cite{BaLZ98,BaLZ00}, establishes the conditions for the existence of an invariant manifold $\cM_\delta$, invariant under the flow of the full system, that remains close to $\cM_0$ (the shortened version of the result presented here is that of Kuehn \cite{Kueh15}; see Theorem 18.2.1).
\begin{theorem}[Bates, Lu, Zeng \cite{BaLZ98}]
\label{teogordo}
Suppose $S_0(t)$ is a semiflow on $X$ and $\cM_0$ is a $C^1$ compact connected manifold, invariant under $S_0(t)$ and normally hyperbolic. Fix $t_1 > t_0$ for some $t_0 > 0$ and let $N_0$ be a sufficiently small tubular neighborhood of $\cM_0$. For $\delta > 0$ sufficiently small there exists $\theta = \theta(\delta) > 0$ such that, if $S_\delta(t)$ is a $C^1$-semiflow satisfying 
\[
\| S_\delta(t_1) - S_0(t_1) \|_1 \leq \theta(\delta), \quad \text{and,} \quad \| S_\delta(t) - S_0(t) \|_0 \leq \theta(\delta), \quad \text{for all } \; \; t \in [0,t_1],
\]
with norms taken with respect to a small neighborhood $B$ of $\cM_0$ such that $N_0 \subset B$, then the semiflow $S_\delta(t)$ has a $C^1$, compact, connected, normally hyperbolic invariant manifold $\cM_\delta$ near $\cM_0$, such that $\cM_\delta$ converges to $\cM_0$ in the $C^1$ topology as $\| S_\delta(t_1) - S_0(t_1) \| \to 0$.
\end{theorem}

\begin{remark}
\label{RemiMafia}
In the present setting, normal hyperbolicity refers intuitively to the property that the flow in directions that are normal to $\cM_0$ dominate the flow in the tangent directions (just like normal hyperbolicity for ODEs \cite{Hek10,Jon95}). The term flow now refers to the semiflow generated by the PDE. See conditions (i)-(iii) in \cite{Kueh15}, section 18.2, for the precise statement and definitions.  
\end{remark}

In the next section, we shall verify that there exists a fast/slow structure associated to system \eqref{HEHnd} for which essentially the same slow manifold studied by Hillen \textit{et al.} \cite{HEH13} is invariant under the fast semiflow associated to the PDE. The key feature is that, under Neumann boundary conditions, the only stationary solutions to the PDE are the equilibrium constant states of the ODE. We gloss over some of the technical details and concentrate on the main properties: existence of the semiflow, invariance of the slow manifold under the fast semiflow, and normal hyperbolicity.

\subsection{Approximation with Neumann boundary conditions}
\label{secNeumi}

Consider the non-local, non-dimensional system \eqref{HEHnd} subject to Neumann boundary conditions \eqref{Neumannbc} and initial conditions of the form $(u,v)(x,0) = (u_0, v_0) \in L^2(\Omega) \times L^2(\Omega)$. We then define the ``fast" system by setting $\delta = 0$. The result is
\begin{equation}
\label{fast}
\begin{aligned}
u_t &= d \Delta u,\\
v_t &= \Delta v - \alpha v + k(p) \int_{\Omega} p(y,t) \, dy,
\end{aligned} \quad \qquad x \in \Omega, \;\; t > 0,
\end{equation}
where $p = u + v$ and subject to boundary conditions of Neumann type $\partial_\nu u = \partial_\nu v = 0$ on $\partial \Omega$.

The first observation is that the non existence of inhomogeneous steady states for system \eqref{fast}, thanks to the boundary conditions of Neumann type. This is a straightforward consequence of the analysis in \cite{Madd14}.

\begin{lemma}
\label{lemuno}
The only stationary solutions $(u,v)(x)$ to the fast system \eqref{fast} under Neumann boundary conditions \eqref{Neumannbc} are constant states.
\end{lemma}
\begin{proof}
Follows directly from the analysis of Maddalena \cite{Madd14} (section 4). (In fact, \eqref{fast} is a particular case of the system considered in \cite{Madd14} with $\delta = 0$; the proof goes verbatim; see also Remark \ref{remXXX}.)
\end{proof}

Consequently, the only equilibrium solutions to \eqref{fast} in $L^2$ are equilibrium constant states of the associated ODE system. Hence, the slow invariant manifold $\cM_0$ under the semiflow of the PDE system \eqref{fast} is essentially the same as in \cite{HEH13}:
\[
\cM_0 = \{ (u,v) \in [0,1] \times [0,v_*(\alpha)] \, : \, M(u,v) = 0 \}.
\]
\[
M(u,v) := k(p) |\Omega| p - \alpha v, \qquad p = u+v.
\]

It is well-known that solutions to the Neumann problem,
\[
w_t = \Delta w, \quad \text{in } \Omega, \qquad \partial_\nu w = 0, \quad \text{on } \partial \Omega,
\]
define an analytic semigroup in $L^2(\Omega)$. In fact, one may consider the Neumann Laplacian $-\Delta_N$ with dense domain $\cD(-\Delta_N) = H^2(\Omega)^\nu := \{ w \in H^2(\Omega) \, : \, \text{tr } \partial_\nu w = 0\}$ acting on $L^2(\Omega)$ as an infinitesimal generator (cf. \cite{HTr08}). Hence, let us consider the operator
\[
\begin{aligned}
A_N \, : \, \cD(A_N) &= H^2(\Omega)^\nu \times H^2(\Omega)^\nu \rightarrow L^2(\Omega) \times L^2(\Omega),\\
A_N \begin{pmatrix} u \\ v \end{pmatrix} &:= - \begin{pmatrix} d & 0 \\ 0 & 1\end{pmatrix} \Delta_N \begin{pmatrix} u \\ v \end{pmatrix}
\end{aligned}
\]
The fast system \eqref{fast} is then recast as
\[
\partial_t \begin{pmatrix} u \\ v \end{pmatrix} = - A_N \begin{pmatrix} u \\ v \end{pmatrix} + \bar{G}(u,v),
\]
where
\[
\bar{G}(u,v) = \begin{pmatrix} 0 \\ G(u,v) \end{pmatrix}, \qquad G(u,v) = -\alpha v + k(p) \int_\Omega p(y,t) \, dy.
\]

To define the fast semiflow associated to solutions to \eqref{fast} we just need to make an observation and to apply standard semigroup theory.
\begin{lemma}
\label{lemdos}
$\bar{G} = \bar{G}(u,v)$ is continuously Fr\'echet differentiable in $L^2(\Omega) \times L^2(\Omega)$.
\end{lemma}
\begin{proof}
It follows from the fact that, for each $h = (h_1, h_2) \in L^2(\Omega) \times L^2(\Omega)$ and each fixed $w = (u,v) \in L^2(\Omega) \times L^2(\Omega)$, there holds $G(w + h) = G(w) + DG(w)h + O(\|h\|_{L^2})$,
where
\[
DG(w)h = -\alpha h_2 + \big( k'(p) \int_\Omega p \, dx \big)(h_1 + h_2) + k(p) \int_\Omega (h_1 + h_2) \, dx,
\]
defines a bounded operator in $L^2$ in view that
\[
\begin{aligned}
\| -\alpha h_2 \|_{L^2} &\leq C_\alpha \| h\|_{L^2 \times L^2},\\
\| (k'(p) \int_\Omega p \, dx \big)(h_1 + h_2) \|_{L^2} &\leq C_{k,p,\Omega} \| h\|_{L^2 \times L^2},\\
\| k(p) \int_\Omega (h_1 + h_2) \, dx \|_{L^2} &\leq C_{k,p,\Omega} \| h\|_{L^2 \times L^2},
\end{aligned}
\]
for each fixed $(u,v) \in L^2(\Omega) \times L^2(\Omega)$.
\end{proof}

As a result, we may apply standard semigroup theory \cite{EN06,Pazy83} to conclude that for each initial condition $(u_0, v_0) \in L^2(\Omega) \times L^2(\Omega)$, the Cauchy problem for system \eqref{fast} subject to Neumann boundary conditions \eqref{Neumannbc} has a unique solution $(u,v) \in C((0,T); \cD(A_N)) \cap C^1([0,T); L^2(\Omega) \times L^2(\Omega))$ for some $T > 0$, which we denote as
\[
(u, v) = S_0(t) (u_0, v_0).
\]
(See, e.g., \cite{Kat75}.) The solution operator $S_0(t)$ is a $C_0$-semigroup (or semiflow) acting on $L^2(\Omega) \times L^2(\Omega)$. It is the \textit{fast semiflow} associated to \eqref{fast} under Neumann boundary conditions.

Thanks to the fact that the only equilibrium states for the fast system \eqref{fast} are constant states we have the following straightforward 
\begin{lemma}
\label{lemtres}
$\cM_0$ is invariant under the fast semiflow $S_0(t)$.
\end{lemma}
\begin{proof}
Take a constant state $(u_M, v_M) \in \cM_0 \subset \cD(A_N)$ (for which $k(p_M) |\Omega| p_M = \alpha v_M$). Then the results follows immediately from uniqueness of the solution $(u,v) = S_0(t) (u_M, v_M)$ to \eqref{fast} with initial condition $(u_M, v_M)$, and from observing that $(u,v) \equiv (u_M, v_M)$ is also a trivial solution to \eqref{fast}. Thus, $S_0(t) \cM_0 = \cM_0$ for all $t \in [0,T]$.
\end{proof}

\begin{lemma}
\label{lemcuatro}
$\cM_0$ is normally hyperbolic with respect to the semiflow $S_0(t)$.
\end{lemma}
\begin{proof}
We need to show that the manifold is attractive for the fast dynamics or solutions to the fast system \eqref{fast}. For that purpose consider $(u_M, v_M) \in \cM_0$ and perturbations $(u+ u_M, v + v_M)$ as solutions to the fast system \eqref{fast}. Linearizing around $(u_M, v_M)$ one obtains
\begin{equation}
\label{linfast}
\begin{aligned}
u_t &= d \Delta u,\\
v_t &= \Delta v - \alpha v + k(p_M) \int_{\Omega} p(y,t) \, dy + k'(p_M) |\Omega| p_M p,
\end{aligned} \quad \qquad x \in \Omega, \;\; t > 0,
\end{equation}
subject to Neumann boundary conditions $\partial_\nu u = \partial_\nu v = 0$ on $\partial \Omega$. Let $\varphi_j$ and $\mu_j$ be the eigenfunctions and eigenvalues, respectively, of the Neumann Laplacian on $\Omega$; that is, $- \Delta \varphi_j = \mu_j \varphi_j$, with $0 = \mu_0 < \mu_1 \leq \ldots \leq \mu_j \leq \ldots$, $\mu_j \to +\infty$, $\partial_\nu \varphi_j = 0$ on $\partial \Omega$ as $j \to +\infty$ and $\varphi_0$ is the constant solution. Therefore, $\mu_1 > 0$ is the first positive eigenvalue, the eigenfunctions are an orthonormal basis of $L^2(\Omega)$ and we can consider expansions of the form
\[
u - u_M = \sum_{j=1}^\infty w_j(t) \varphi_j(x), \qquad v - v_M = \sum_{j=1}^\infty y_j(t) \varphi_j(x).
\]
Upon substitution into \eqref{linfast} one finds that for each $j$,
\[
\frac{d}{dt} \begin{pmatrix} w_j \\ y_j  \end{pmatrix} = (- \mu_j D + J_M) \begin{pmatrix} w_j \\ y_j  \end{pmatrix}
\]
with
\[
D = \begin{pmatrix} d & 0 \\ 0 & 1\end{pmatrix}, \qquad J_M = \begin{pmatrix} 0 & 0 \\ |\Omega| (k'(p_M) p_M + k(p_M) )& |\Omega| (k'(p_M) p_M + k(p_M)) - \alpha \end{pmatrix}.
\]
The eigenvalues of $- \mu_j D + J_M$ are 
\[
\lambda^{(j)}_1 = - \mu_j d, \qquad \lambda^{(j)}_2 = |\Omega|k'(p_M) p_M - \mu_j + |\Omega|k(p_M) - \alpha.
\]
Clearly $\lambda^{(j)}_1 < 0$ for all $j \geq 1$. Substituting $k(p_M) |\Omega| p_M = \alpha v_M$ we notice that
\[
\lambda^{(j)}_2 =  |\Omega|k'(p_M) p_M - \mu_j - |\Omega| k(P_M) u_M < 0,
\]
inasmuch as $k$ is decreasing. Thus, solutions to the linearized fast system around a point in $\cM_0$ decay and time and we conclude that the slow manifold is normally hyperbolic.
\end{proof}

By Theorem \ref{teogordo}, for $0 < \delta \ll 1$ sufficiently small there exists an invariant manifold $\cM_\delta$ under the semiflow $S_\delta(t)$ associated to system \eqref{HEHnd}, which is close to $\cM_0$. The long time dynamics of solutions to \eqref{HEHnd} are therefore settled onto the slow manifold $\cM_0$, which can be expressed as a graph (see Lemma 3.2 in \cite{HEH13}). Indeed, by the implicit function theorem, for every $(u_M, v_M) \in \cM_0$ satisfying $M(u_M , v_M) = 0$, and since
\[
\partial_v M(u_M, v_M) = k'(p_M) |\Omega| p_M - k(p_M) |\Omega| u_M < 0,
\]
then the slow manifold can be recast as $\cM_0 = \{ (u, \widetilde{v}_\alpha(u) \, : \, u\in [0,1] \}$, where
\[
\frac{d \widetilde{v}_{\alpha}}{du} = \frac{|\Omega| (k'(p) p + k(p))}{\alpha - |\Omega| (k'(p) p + k(p))}, \qquad p = u + \widetilde{v}_{\alpha}(u).
\]
Moreover, it can be shown that for $\alpha_1 > \alpha_2$, one has $\widetilde{v}_{\alpha_1}(u) < \widetilde{v}_{\alpha_2}(u)$(see \cite{HEH13}, Lemma 3.2).

As a result, the tumor growth paradox arises due to the properties of the slow manifold analyzed by Hillen \textit{et al.} \cite{HEH13}. (Actually, the analysis can be quoted word by word.) We omit the details and simply state the emergence of the paradox in the current PDE setting. Notice that we need to assume that the initial conditions are near the attracting slow manifold $\cM_0$ to guarantee the emergence of the paradoxical tumor growth.
\begin{theorem}
Under the assumptions \eqref{hypk} on the progeny kernel $k = k(p)$ and for $0 < \delta \ll 1$ sufficiently small, let $(u_1, v_1)$ and $(u_2, v_2)$ be the corresponding solutions to the PDE system \eqref{HEHnd} for values of $\alpha_1$ and $\alpha_2$, respectively, under Neumann boundary conditions \eqref{Neumannbc} with the same initial condition $(u,v)(x,0)$ (same initial tumor size) near the slow manifold $\cM_0$. Assume that $\alpha_1 > \alpha_2 > 0$. Then there exists times $t_a, t_b$ such that 
\[
\int_{\Omega} p_1(x,t_a) \, dx = \int_\Omega p_2(x, t_b) \, dx, \quad \text{and } \quad \int_{\Omega} p_1(x,t_a + \theta) \, dx > \int_\Omega p_2(x, t_b + \theta) \, dx,
\]
for each $\theta > 0$.
\end{theorem}


\section*{Acknowledgements}


The authors are grateful to Thomas Hillen, for calling their attention to the work by Delgado \textit{et al.} \cite{DeDS18} and for enlightening conversations, as well as to Alessio Franci, for pointing out reference \cite{Kueh15}.  This work is an extension of the M. Sc. Thesis of the first author \cite{PadTh19}, written under the supervision of the second, and was partially supported by DGAPA-UNAM, program PAPIIT, grant IN-100318.

\def\cprime{$'$}

%



\end{document}